\documentclass[12pt,a4paper,twoside,reqno]{amsart}
\title[Generalized connections and Heterotic Supergravity]{Torsion-free generalized connections and Heterotic Supergravity}
\author[M. Garcia-Fernandez]{Mario Garcia-Fernandez}
  \address{\'Ecole Polytechnique F\'ed\'eral de Lausanne\\ EPFL SB MATHGEOM GEOM\\
    MA B1 437, Station 8\\ CH-1015 Lausanne, Switzerland.}
  \email{mario.garcia@epfl.ch}

\thanks{The author is currently supported by the EPFL (\'Ecole Polytechnique F\'ed\'eral de Lausanne). The initial work was supported by QGM (Centre for Quantum Geometry of Moduli Spaces), funded by the Danish National Research Foundation.} %Part of this work was undertaken during the author's visit to the Mathematical Institute in Oxford.}

\usepackage{hyperref,url}
\usepackage{amssymb,latexsym}
\usepackage{amsmath,amsthm}
\usepackage[latin1]{inputenc}
\usepackage{enumerate}
\usepackage[all]{xy} \CompileMatrices
\SelectTips{cm}{12} % computer modern fonts of size 12 for arrowheads
\theoremstyle{plain}
\newtheorem{theorem}{Theorem}[section]
\newtheorem{lemma}[theorem]{Lemma}

\newtheorem{proposition}[theorem]{Proposition}

\theoremstyle{definition}
\newtheorem{definition}[theorem]{Definition}
\newtheorem{definition-theorem}[theorem]{Definition-Theorem}

\theoremstyle{remark}
\newtheorem{remark}[theorem]{Remark}

\newcommand{\secref}[1]{\S\ref{#1}}

\numberwithin{equation}{section} 

\newcommand{\tr}{\operatorname{tr}}
\newcommand{\Id}{\operatorname{Id}}

\newcommand{\End}{\operatorname{End}}

\newcommand{\Ker}{\operatorname{Ker}}
\newcommand{\ad}{\operatorname{ad}}
\newcommand{\Ad}{\operatorname{Ad}}
\newcommand{\Aut}{\operatorname{Aut}}
\newcommand{\RR}{{\mathbb R}}
\newcommand{\ZZ}{{\mathbb Z}}
\newcommand{\rk}{\operatorname{rk}}

\renewcommand{\(}{\left(}
\renewcommand{\)}{\right)}

\newcommand{\Vol}{\operatorname{Vol}}
\newcommand{\surj}{\to\kern-1.8ex\to}

\begin{document}

\maketitle

\begin{abstract}%
This work revisits the notions of connection and curvature in generalized geometry, with emphasis  on torsion-free generalized connections on a transitive Courant algebroid. % , compatible with a generalized metric. % The precise geometric set-up for our study goes beyond Hitchin's framework for generalized complex geometry, making use of transitive Courant algebroids. 
%Geometric structures on non-exact Courant algebroids have been recently considered by Rubio in the context of $B_n$-generalised geometry, as proposed by Baraglia. % have been considered recently by R. Rubio in the context of $B_n$-generalised geometry and arise naturally from the theory of generalized reduction of Burzstyn, Cavalcanti and Gualtieri. 
As an application, we provide a mathematical derivation of the equations of motion of heterotic supergravity in terms of the Ricci tensor of a generalized metric, inspired by the work of Coimbra, Strickland-Constable and Waldram.% in type II and $11$-dimensional supergravity.

\end{abstract}

%\tableofcontents

\section{Introduction}

Generalized connections were introduced by M. Gualtieri in \cite{G3}, for the study of holomorphic Poisson structures and generalized K\"ahler geometry. In this work we revisit this notion %per se
with a completely different emphasis: we fix a generalized metric and consider generalized connections compatible with the metric and with vanishing torsion. The study of torsion-free generalized connections has not been addressed before in the mathematics literature, probably due to the uneasy fact that, in generalized geometry, the metric does not determine a unique such connection. In addition, the precise geometric set-up for our study makes use of transitive Courant algebroids \cite{LiuWeXu,Severa,Vaisman,Bressler}, extending Hitchin's original proposal for generalized complex geometry \cite{Hit1}. %Roughly, a transitive Courant algebroid is a vector bundle whose sheaf of sections is endowed with a bracket and a symmetric product and a canonical projection to vector fields on the base manifold $M$. 
Non-exact Courant algebroids arise naturally from the theory of reduction of Burzstyn, Cavalcanti and Gualtieri \cite{BuCaGu}. In generalized geometry, its study was proposed by Baraglia \cite{Ba} and has been recently developed by Rubio in \cite{Rubio}. % in the context of $B_n$-generalised geometry. Furthermore, non-exact Courant algebroids arise naturally from the theory of generalized reduction of Burzstyn, Cavalcanti and Gualtieri \cite{BuCaGu}. As an application, we provide a mathematical derivation of the equations of motion of heterotic supergravity \cite{BR} (see also \cite{BKO,FIVU}), inspired by the geometric description of $11$-dimensional and type II supergravity by Coimbra, Strickland-Constable and Waldram \cite{CSW,CSW2}.

One motivation for the study of generalized geometry is that it %appears to 
provides a natural framework to describe the geometry of string theory or M-theory in the presence of fluxes \cite{Hull3}. %, or at least of its supergravity low energy limit. 
In Hitchin's seminal work \cite{Hit1}, generalized geometry is formulated in the generalized tangent bundle $TM \oplus T^*M$ of a manifold $M$ and incorporates a closed $3$-form that can be interpreted as the $H$-flux in type II string theory. For the study of $M$-theory, this structure was extended by Hull \cite{Hull3} considering higuer order forms and a modification of the generalized tangent bundle. Despite its interest in physics, a similar treatment of the heterotic string has been, so far, elusive. The first difficulty is that the $H$-flux in this theory is not closed, but rather satisfies the \emph{Bianchi identity}
\begin{equation}\label{eq:anintro}
dH = \alpha'\(\tr R\wedge R - \tr F_A \wedge F_A\),
\end{equation}
relating the $4$-form $dH$ with the curvatures $R$ and $F_A$ of a pair of connections; one in the tangent bundle, the other in a principal (gauge) bundle over $M$ (see Section \ref{sec:sugra} for details). 

As an application of our study of torsion-free generalized connections, in this work we provide a mathematical derivation of the equations of motion of heterotic supergravity \cite{BR} (see also \cite{BKO,FIVU}) -- the low energy limit of the heterotic string -- inspired by the geometric description of $11$-dimensional and type II supergravity by Coimbra, Strickland-Constable and Waldram \cite{CSW,CSW2}. Crucially, the Bianchi identity \eqref{eq:anintro} is incorporated in the geometry of a transitive Courant algebroid over $M$ via the definition of its bracket. This partially confirms Bouwknegt's recent proposal \cite{Bouw} for the geometrization of supergravity theories using generalized geometry (for previous attempts in the heterotic case see \cite{Andriot,HoKw}) 

To explain the geometric set-up, in Section \ref{sec:reduction} we construct transitive Courant algebroids via generalized reduction \cite{BuCaGu}. Given a principal $G$-bundle $P$ with vanishing first Pontryagin class, in Proposition \ref{lemma:canonical} we observe that its generalized tangent bundle $TP \oplus T^*P$ % its total space 
carries a canonical structure of (exact) Courant algebroid endowed with a (lifted) $G$-action. Using generalized reduction, we construct a transitive Courant algebroid $E$ with underlying vector bundle
$$
TM \oplus \ad P \oplus TM^*
$$
over the base manifold $M = P/G$ and calculate explicitely its bracket. %For this, in Lemma \ref{lemma:preferred} we derive an explicit expression of the closed three form and the lifted action on an exact Courant algebroid over a principal bundle. 
To the knowledge of the author, the construction of transitive Courant algebroids by reduction was first pointed out by Severa in \cite{Severa}. A more fundamental approach to these objects in the context of generalized geometry has been undertaken by Rubio in \cite{Rubio} (for abelian group $G$).

In Section \ref{sec:metricsconnec} we study torsion-free generalized connections compatible with a generalized metric, on the reduced space $E$. We define a notion of generalized metric $V_+ \subset E$, that we call admissible, and prove tha existence of a canonical compatible generalized connection $D^0$ with vanishing torsion. A remarkable fact about $D^0$ is that it singles out four (standard) connections with skew torsion, compatible with a metric $g$, given by (see \eqref{eq:D-+})
\begin{align*}
\nabla^\pm &= \nabla^g \pm \frac{1}{2}g^{-1}H,\\
\nabla^{\pm1/3} &= \nabla^g \pm \frac{1}{6}g^{-1}H,
\end{align*}
where $\nabla^g$ denotes the Levi-Civita connection of $g$. This set of connections was used by Bismut \cite{Bismut} to prove a Lichnerowicz type formula and a local index theorem for connections with skew torsion (see also \cite{Ferreira,KiYi}). 

The assignment $V_+ \to D^0$ is naturally preserved by the symmetries of $E$ and, based on this, we believe that the connection $D^0$ is \emph{the} analogue of the Levi-Civita connection in generalized geometry (see Remark \ref{rem:LeviCivita}). However, as pointed out in \cite{CSW}, we should emphasize that a choice of generalized metric $V_+$ does not determine uniquely a torsion-free compatible generalised connection. Far from that, one can modify the connection $D^0$ by elements in the kernel of the natural map
\begin{equation}\label{eq:naturalmap}
(V_+ \otimes \Lambda^2V_+) \oplus (V_- \otimes \Lambda^2V_-) \to \Lambda^3V_+ \oplus \Lambda^3V_-,
\end{equation}
where $V_-$ denotes the orthogonal of $V_+$ with respect to the ambient metric on $E$. To explicitely exhibit this fact -- and as a need for our applications -- we construct a family of such connections $D^\varphi$ parameterized by $1$-forms $\varphi \in \Omega^1(M)$, adding a \emph{Weyl term} to the canonical connection $D^0$.

In Section \ref{sec:Ricci} we study two natural quantities introduced by M. Gualtieri \cite{G3} that are canonically attached to any generalized connection, namely, the generalized curvature $\operatorname{GR}$ and the generalized Ricci tensor $\operatorname{GRic}$ (see also \cite{CSW}). We provide explicit formulae for these tensors for any element in our family $D^\varphi$. This lead us to our main application: for a ten dimensional spin manifold $M$ and a suitable choice of principal bundle $P$, in Section \ref{sec:sugra} we prove that the combination of the equations of motion and the Bianchi identity in heterotic supergravity coincide with the Ricci flat condition (see Theorem \ref{th:het})
$$
\operatorname{GRic} = 0,
$$
for the generalized connection $D^\varphi$ on the transitive Courant algebroid $E$. Here, the choice of one form
$$
\varphi = \frac{-6}{\rk V_+ - 1} d\phi
$$
corresponds to the choice of a dilaton field $\phi \in C^\infty(M)$ in the physical theory. For this, we note that the other (bosonic) fields of the theory, provided by a metric $g$, a three form $H\in \Omega^3(M)$ and a gauge connection $A$, are naturally parameterized by the notion of admissible metric $V_+$. %Crucially, the anomaly equation
%\begin{equation}\label{eq:anintro}
%dH = \alpha'\(\tr R\wedge R - \tr F_A \wedge F_A\)
%\end{equation}
%is incorporated in the geometry of $E$ via the definition of its bracket. In \eqref{eq:anintro}, the right hand side is a representant of the first Pontryagin class of the bundle $P$ scaled by a physical parameter $\alpha'\in \RR$ (see Section \ref{sec:sugra} for details). %$\alpha'$ is a real parameter and $R$ and $F_A$ correspond to (pieces of) the curvature of a connection on $P$.
%We should mention that it has been pointed before by physicists \cite{Bouw} %,HoKw}
%that transitive Courant algebroids provide an appealing framework to understand the underlying differential geometry of the heterotic string, although apparently this idea has not been explored before in the literature.

To clarify ideas an exhibit the main differences of our approach with that of \cite{CSW}, we first discuss how our framework applies to Type II supergravity in Section \ref{subsec:typeII}. A novelty with respect to \cite{CSW,CSW2} is the treatment of the dilaton field. In this work, we interpret the freedom provided by the dilaton field in the physical theory as the freedom of choice of torsion-free metric connection in generalized geometry. This opens the possibility of considering more general fields, in the kernel of the natural map \eqref{eq:naturalmap} playing the role of the dilaton in supergravity and, in particular, in compactifications of the heterotic string. We leave this perspective to the physicists. 

An interesting fact of our main result, as well as in the generalized geometric treatment of type II \cite{CSW} and $11$-dimensional supergravity \cite{CSW2}, is that the $3$ natural quantities which arise from the generalized Ricci tensor agree exactely with the leading order term in $\alpha'$-expansion of the $\beta$ functions in the %heterotic 
sigma model approach \cite{CFMP}. This
%The vanishing of the $\beta$ functions is equivalent to the conformal invariance of the associated quantum field theory, which
stablishes an identification between the renormalization group flow of these physical theories and the \emph{generalized Ricci flow} (see \cite{O,StTi,StTi2} for work done in this direction). %intriguing link between generalized geometry and conformal field theory (and renormalization theory).

Further motivation for this work comes from the complex geometry of heterotic compactifications, provided by the Strominger System \cite{Strom}. From a mathematical perspective, this system provides a generalization of the K\"ahler Ricci-flat equation for the case of non-K\"ahler Calabi-Yau manifolds and its study has been proposed by S.-T. Yau in order to understand moduli spaces of complex $3$-folds with trivial canonical bundle. Although there has been recent progress in the study of the existence of solutions (see \cite{AGF1,AGF2,LiYau} and references therein), this problem is widely open, the main difficulties being its non-K\"ahler nature and the lack of understanding of its geometry. Throughout this paper we hope to show that generalized geometry provides a promising approach to study the geometry arising from the heterotic string. The application of this theory to the Strominger System will require the understanding of the supersymmetry variations in terms of spinors in transitive Courant algebroids and the study of holomorphic structures on these objects. We hope to address these and other related questions in future work.

%In \S \ref{sec:metricconnection} we study generalized geometry on the reduced Courant algebroid $E$ when $G$ is a real semi-simple Lie group, focusing on generalized metrics, torsion free connections and curvature, following \cite{CSW,G3}. Our main result is an explicit formula for the generalized Ricci tensor of a suitable type of generalized metrics. The results in this section add to the work of R. Rubio in \cite{Rubio}, who studied generalized geometry on $T \oplus 1 \oplus T^*$, corresponding to the case of abelian group $G$.

%Generalized connections and its torsion were introduced by M. Gualtieri in \cite{G3}, to highlight new aspects of generalized K\"ahler geometry. More recently, torsion free connections compatible with a generalized metric have been the object of further studies in \cite{CSW}. Our task in this section is to define a notion of admissible metric on the reduced Courant algebroid
%\begin{equation}\label{eq:exactcourant}
%0 \to T^* \to E \to T \to 0,
%\end{equation}
%and to construct a compatible torsion free generalized connection in a canonical way. As pointed out in \cite{CSW}, torsion free metric connection are not unique, but the quantities we are ultimately interested in, the Ricci tensor and the scalar curvature, do not depend on the choice of a particular connection.

\vspace{12pt}

\noindent
{\bf Acknowledgements.}
I am grateful to Nigel Hitchin for stimulating discussions and insight during my visit to the Mathematical Institute. Thanks also to Roberto Rubio for many `generalized hours' and his encouragement with this project. I wish to thank L. Alvarez-Consul, Bjorn Andreas, Marco Gualtieri and James Sparks for helpful discussion, and L. Alvarez-Consul and O. Garcia-Prada for their inspiring teaching on Kaluza-Klein theory. Thanks to the Centre for Quantum Geometry of Moduli Spaces for the support and confidence to undertake this project and to the Mathematical Institute in Oxford for the hospitality.

\section{Transitive Courant algebroids and reduction}\label{sec:reduction}

In this section we study lifted actions \cite{BuCaGu} on an exact Courant algebroid over a principal bundle. Assuming the existence of an equivariant isotropic splitting, we will show that any connection determines an explicit presentation of the curvature $3$-form in terms of the Chern Simons $3$-form, up to $B$-field transformation and basic terms. Conversely, starting with a principal $G$-bundle $P$ with vanishing first Pontryagin class, we observe that its total space carries a canonical exact Courant algebroid endowed with a lifted action and admitting an equivariant isotropic splitting. As observed by Severa \cite{Severa}, when the induced quadratic form $c$ on the Lie algebra of the structure group is non-degenerate, the action is non-isotropic and lead us to a (non-exact) transitive Courant algebroid over the quotient, as considered in \cite{ChStXu}.

\subsection{Lifted actions on a principal bundle}

Let $G$ be a real Lie group with Lie algebra $\mathfrak{g}$. Let $P$ be a smooth principal $G$-bundle over a smooth $n$-dimensional manifold $M$, with action on the right. Let $\hat E$ be an exact Courant algebroid over $P$, given by a short exact sequence
\begin{equation}
\label{eq:shortexactCourant}
\begin{gathered}
  \xymatrix{ & 0 \ar[r] & TP \ar[r]^-{\pi^*} & \hat E \ar[r]^-{\pi} & TP \ar[r] & 0.}
\end{gathered}
\end{equation}
Let $\psi\colon \mathfrak{g} \to \Gamma(TP)$ be the Lie algebra homomorphism given by the infinitesimal $\mathfrak{g}$-action on $P$.

\begin{definition}[\cite{BuCaGu}]
A lift of the $G$-action to $\hat E$ is an algebra morphism $\rho\colon \mathfrak{g} \to \Gamma(\hat E)$ making commutative the diagram
\begin{equation}
\label{eq:extendeddiagram}
\begin{gathered}
  \xymatrix{\mathfrak{g} \ar[r]^-{\rho} \ar[dr]_-{\psi} & \Gamma(\hat E) \ar[d]^{\pi} \\  & \Gamma(TP)}
\end{gathered}
\end{equation}
and such that the infinitesimal $\mathfrak{g}$-action on $\Gamma(\hat E)$ induced by the Courant bracket integrates to a (right) $G$-action on $\hat E$ lifting the action on $P$.
\end{definition}

In the sequel, we fix a lifted $G$-action $\rho\colon \mathfrak{g} \to \Gamma(\hat E)$ on $\hat E$ and assume that it is $G$-equivariant, that is
$$
\rho(\Ad(g)z) = g\rho(z)
$$
for all $g \in G$ and $z \in \mathfrak{g}$. Furthermore, we assume that the action on $\hat E$ admits at least one equivariant isotropic splitting $\lambda\colon TP \to \hat E$.

Associated to the lifted action $\rho$ there is an invariant quadratic form on $\mathfrak{g}$, defined by
\begin{equation}\label{eq:quadratic}
c(z) = - \langle \rho(z),\rho(z)\rangle.
\end{equation}
Given an equivariant isotropic splitting $\lambda\colon TP \to \hat E$ of \eqref{eq:shortexactCourant}, we obtain an isomorphism
$$
\lambda + \frac{1}{2}\pi^*\colon TP \oplus T^*P \to \hat E
$$
inducing a $G$-invariant, closed, curvature $3$-form $\hat H \in \Omega^3(P)^G$, defined by
$$
\hat H(X,Y,Z) = \langle[\lambda X, \lambda Y],\lambda Z\rangle
$$
and a (Dorfman) bracket given by
$$
[X + \xi,Y+\eta] = L_X(Y+\eta) - i_Y d\xi + i_Yi_X \hat H.
$$
The equivariance of the splitting and of the transported lifted action
\begin{equation}\label{eq:extendedaction}
\rho(z) = Y_z + \xi_z, \qquad z \in \mathfrak{g}
\end{equation}
can be written respectively as
\begin{equation}\label{eq:extactionpropert}
d\xi_z = i_{Y_z} \hat H \qquad \textrm{and} \qquad \xi_{\Ad(g)z} = R_g^*(\xi_z).
\end{equation}

For any choice of connection $A\colon TP \to VP$ on $P$, the following lemma provides a preferred presentation of the lifted action and the curvature $\hat H$ up to a $B$-field transformation. Similar ideas have been considered previously in \cite[Ex. 6.2]{CaGu}, but the following result does not seem to be in the literature. Consider the Chern-Simons $3$-form
$$
CS(A) = - \frac{1}{6} c(A\wedge [A,A]) + c(F\wedge A) \in \Omega^3(P)^G,
$$
where $F = F_A \in \Omega^2(P,VP)^G$ denotes the curvature of the connection
$$
F(X,Y) = -A[A^\perp X,A^\perp Y].
$$
Recall that $CS(A)$ is defined such that its differential equals the representative of the first Pontriagin class of the bundle
$$
d CS(A) = c(F\wedge F),
$$
and, more explicitely, it can be written as
\begin{equation}\label{eq:CSA}
\begin{split}
CS(A)(X,Y,Z) & = - c(AX,[AY,AZ])\\
& + c(F(X,Y),AZ) - c(F(X,Z),AY) + c(F(Y,Z),AX)
\end{split}
\end{equation}
for invariant vector fields $X,Y,Z \in \Gamma(TP)^G$.

\begin{lemma}\label{lemma:preferred}
For any connection $A\colon TP \to VP$ there exists an equivariant isotropic splitting $\lambda\colon TP \to \hat E$ and a $3$-form $H \in \Omega^3(M)$ such that
\begin{equation}\label{eq:extendedactpreferred}
\rho(z) = Y_z - c(z,A\cdot),
\end{equation}
\begin{equation}\label{eq:Hpreferred}
\hat H = -CS(A) + p^*H.
\end{equation}
This splitting is unique up to basic $B$-field transformation, that is, transformation by pull-back of a $2$-form $b \in \Omega^2(M)$
\end{lemma}
\begin{proof}
Define an invariant two form $b \in \Omega^2(P)^G$ by $b = b' - \hat b$, where
\begin{align*}
b'(Y_1,Y_2) & = i_{Y_2}\xi_{A Y_1} - i_{Y_1}\xi_{A Y_2}\\
\hat b(Y_1,Y_2) & = \frac{1}{2}(i_{A Y_2}\xi_{A Y_1} - i_{AY_1}\xi_{AY_2}),
\end{align*}
and note that
$$
\xi_z(Y) = b(Y_z,Y) - c(z,AY).
$$
Then, in the equivariant splitting $\lambda' = e^{b}\lambda$ the lifted action is given by
$$
\rho(z) = Y_z + \xi'_z = Y_z + \xi_z - i_{Y_z}b = Y_z - c(z,A\cdot)
$$
which proves \eqref{eq:extendedactpreferred}. To prove \eqref{eq:Hpreferred}, it is enough to check that
$$
i_{Y_z}(\hat H + db) = -i_{Y_z}CS(A),
$$
for any $z \in \mathfrak{g}$. Note that $i_{Y_z}(\hat H + db) = d\xi'_z$ by \eqref{eq:extactionpropert} and hence, using that
$$
d\xi'_z(X,Y) = - X(c(z,AY)) + Y(c(z,AX)) + c(z,A[X,Y])
$$
for any invariant vector fields $X,Y$ on $P$ we obtain
\begin{align*}
d\xi'_z(Y_{z'},Y_{z''}) & = c(z,[z',z''])\\
d\xi'_z(Y_{z'},A^\perp Y)& = c(z,[Y_{z'},A^\perp Y]) = 0\\
d\xi'_z(A^\perp X,A^\perp Y)& = c(z,A[A^\perp X,A^\perp Y]) = - c(F(X,Y),z),\\
\end{align*}
which proves \eqref{eq:Hpreferred}. The statement about uniqueness is obvious.
\end{proof}

Our discussion lead us to the following converse of the previous result, that shall be compared with the main result in \cite{ChStXu}.
\begin{proposition}\label{lemma:canonical}
Given an invariant quadratic form $c$ on $\mathfrak{g}$ and a principal $G$-bundle $P$ with vanishing first Pontryagin class, there exists a canonical exact Courant algebroid over $P$, endowed with a lifted action and admitting an equivariant isotropic splitting, uniquely determined up to isomorphism.
\end{proposition}
For the proof, one simply makes a choice of connection and defines a bracket on $TP \oplus T^*P$ using the closed three form \eqref{eq:Hpreferred}. The lifted action is then given by \eqref{eq:extendedactpreferred} and the independence of the choices up to isomorphism follows from the proof of the previous lemma. Note that we have the following structure equation for the exterior differential of the $3$-form $H \in \Omega^3(M)$ in \eqref{eq:Hpreferred}
\begin{equation}\label{eq:dH}
dH = c(F\wedge F).
\end{equation}

\subsection{Non-isotropic lifted actions}

We consider now the case when $c$ is non-degenerate, in the context of Proposition \ref{lemma:canonical}. Using the lifted action, we can construct a reduced Courant algebroid $E$ over $M= P/G$. For this, let
$$
K := \rho(\mathfrak{g}) \subset \hat E
$$
and $K^\perp$ be its orthogonal on $\hat E$. Then, by \cite[Lemma 3.2]{BuCaGu}, $K$ and $K\cap K^\perp$ have constant rank and hence define $G$-equivariant vector bundles over $P$. Therefore, we can define $E$ as the quotient vector bundle (see equation $(15)$ in \cite{BuCaGu})
\begin{equation}
E = E_{red} = \frac{K^\perp}{K\cap K^\perp}\Big{/} G
\end{equation}
over $M$. As a direct aplication of \cite[Theorem 3.3]{BuCaGu}, we obtain that $E$ inherits a structure of Courant algebroid over $M$ with surjective anchor, that is, a transitive Courant algebroid.

As the image of the lifted action $K$ is not isotropic, the reduced space $E$ is no longer an exact Courant algebroid, as we will see explicitely in the next proposition. In the sequel, we denote by $T$ the tangent bundle of $M$ and by $\nabla^A$ the connection on $\ad P$ induced by a connection $A$, explicitely given by (see \cite{AB})
$$
\nabla^A_X(r) = [A^\perp X,r]
$$
for $X \in \Omega^0(T)$ and $r \in \Omega^0(\ad P)$.

\begin{proposition}\label{prop:reducedcourant}
Suppose that $c$ is non-degenerate. % and that $\hat E$ admits an equivariant splitting.
Then, any connection $A$ determines an isomorphism
$$
E \cong TP/G \oplus T^*
$$
and a $3$-form $H$ on $M$, uniquely up to basic exact $3$-forms, such that the symmetric pairing is given by
$$
\langle X + \xi,Y + \eta\rangle = \frac{1}{2}(i_{pX}\eta + i_{pY}\xi) + c(AX,AY)
$$
and the Dorfman Bracket is given by
\begin{equation}\label{eq:bracket}
\begin{split}
[X+\xi,Y+\eta] & = [X,Y] + L_{pX}\eta - i_{pY}d\xi + i_{pY}i_{pX}H\\
& + 2c(\nabla^A(A X),A Y) + 2c(F(X,\cdot),A Y) - 2c(F(Y,\cdot),A X).
\end{split}
\end{equation}
\end{proposition}
\begin{proof}
By Lemma \ref{lemma:preferred}, given $A$ there exists an equivariant isotropic splitting such that
$$
\rho(z) = Y_z - c(z,A\cdot),
$$
uniquely defined up to $B$-transformation by $p^*b$, for $b \in \Omega^2(M)$. Then, a direct calculation shows that
$$
K^\perp = \{X + c(A X,A\cdot) + \xi \; : \; X \in \Gamma(TP), \xi \in \Omega^1(P)_{bas}\},
$$
where $\Omega^1(P)_{bas}$ denotes the space of basic $1$-forms in $P$. Hence,
$$
K \cap K^\perp = \{0\}
$$
and the sections of $E$ are given by invariant sections of $K^\perp$, which are canonically identified with
$$
\Gamma(TP/G \oplus T^*M) \cong \Gamma(K^\perp)^G \colon X + \xi \to X + c(A X,A\cdot) + p^*\xi .
$$
The statement is now straightforward by direct calculation of the Courant bracket
\begin{align*}
[X+\xi,Y+\eta] & = [X,Y] + L_{pX}\eta - i_{pY}d\xi\\
& + L_{X}(c(A Y,A\cdot)) - i_Yd(c(A X,A\cdot)) + i_Yi_X(p^*H- CS(A))
\end{align*}
and the symmetric pairing on $\Gamma(K^\perp)^G$.
\end{proof}

To compare the previous result with the analysis of regular Courant algebroids in \cite{ChStXu}, note that any connection $A$ provides further an isomorphism
\begin{equation}\label{eq:splittingA}
E \cong T \oplus \ad P \oplus T^*,
\end{equation}
which endows $T \oplus \ad P \oplus T^*$ with a standard structure of transitive Courant algebroid (see \cite[\S 2]{ChStXu}).

\section{Admissible metrics and torsion free connections}\label{sec:metricsconnec}

In this section we study generalized geometry on the reduced Courant algebroid
\begin{equation}\label{eq:exactcourant}
0 \to T^* \to E \to T \to 0,
\end{equation}
provided by Propositions \ref{lemma:canonical} and \ref{prop:reducedcourant}, focusing on generalized metrics and torsion free connections, following \cite{CSW,G3}. We define a notion of generalized metric on $E$ that we call admissible and construct a compatible torsion free generalized connection $D^0$ in a canonical way. As pointed out in \cite{CSW}, torsion free metric connections in generalized geometry are not unique, and this enables us to consider deformations of $D^0$ adding a \emph{Weyl term}.

Although $E$ and $T \oplus \ad P \oplus T^*$ are isomorphic, there exists no canonical isomorphism, so we adopt a splitting independent approach. Generalized connections and its torsion were defined by M. Guatieri in \cite{G3}. Here we follow closely his approach. We rely on previous results on generalized geometry for transitive Courant algebroids by R. Rubio in \cite[Section 2.1]{Rubio}.

\subsection{Preliminaries: Linear geometry}

In this section we recall the linear geometry of a fibre $W = E_x$ of the reduced Courant algebroid \cite{ChStXu,Rubio}, to introduce the neccessary notation and conventions.

Let $W$ and $V^*$ be real vector spaces, defining a complex
\begin{equation}
\label{eq:linearsurjection}
\begin{gathered}
  \xymatrix{V^* \ar[r] &  W \ar[r]^-{\pi} & V \ar[r] & 0,}
\end{gathered}
\end{equation}
for a surjective map $\pi$. We assume that $W$ is endowed with a metric $\langle\cdot,\cdot\rangle$ of arbitrary signature such that the image of $V^* \to W^*$ is isotropic. Furthermore, we assume that the first arrow is given by the composition of $\frac{1}{2}\pi^*\colon V^* \to W^*$ with the isomorphism $W \cong W^*$ given by the metric. Then, we have canonical flags
\begin{equation}\label{eq:flag1}
V^* \subset \Ker \pi \subset W,
\end{equation}
\begin{equation}\label{eq:flag2}
\mathfrak{g} = \Ker \pi/V^* \subset Q = W/V^*.
\end{equation}
Moreover, $V^* = \Ker \pi^\perp$ and hence $\mathfrak{g}$ inherits a metric $c$ from $W$, that we assume is non degenerate.

Given an isotropic splitting $\lambda\colon V \to W$ of \eqref{eq:linearsurjection} we can construct an isomorphism
\begin{equation}\label{eq:split}
W \cong V \oplus \mathfrak{g} \oplus V^*
\end{equation}
such that the transported metric is given by
$$
\langle X + r + \xi, X + r + \xi\rangle = \xi(X) + c(r,r).
$$
More explicitely, if we denote by $\lambda^\perp$ the orthogonal to the image of $\lambda$, the induced isomorphism is
\begin{equation}\label{eq:Psilambda}
\Psi_\lambda\colon V \oplus \mathfrak{g} \oplus V^* \to W \colon X + r + \xi \to \lambda(X) + \pi_{Q|\lambda^\perp}^{-1}(r) + \xi,
\end{equation}
where $\pi_Q\colon W \to Q$ is the natural projection and $\pi_{Q|\lambda^\perp}$ is the isomorphism given by its restriction to $\lambda^\perp$.

The space of isotropic splittings of \eqref{eq:linearsurjection} is an affine space modelled on the vector space $\Omega^2 \oplus \Omega^1(\mathfrak{g})$, where, given a non-negative integer $k$, we use the notation
$$
\Omega^k = \Lambda^k V^*, \qquad \Omega^k(\mathfrak{g}) = \Omega^k \otimes \mathfrak{g}.
$$
To see this, note simply that any other isotropic splitting can be constructed as
$$
\lambda_{a,b} = \lambda + \pi_{Q|\lambda^\perp}^{-1}a + b - c(a\otimes a),
$$
for $(b,a) \in \Omega^2 \times \Omega^1(\mathfrak{g})$. To calculate the change in the isomorphism \eqref{eq:Psilambda}, recall from \cite{Rubio} that in the splitting \eqref{eq:split} a general element of the Lie algebra $\mathfrak{so}(W)$ of orthogonal symmetries can be written as a block matrix, given by
\begin{equation}
\label{eq:Liealg} \left (\begin{array}{ccc}
f & - 2c(\alpha,\cdot) & \beta \\
a & e & \alpha\\
b & - 2c(a,\cdot) & - f^T
\end{array}\right ),
\end{equation}
where $f \in \End V$, $b \in \Omega^2$,  $a \in \Omega^1(\mathfrak{g})$, $\beta \in \Lambda^2V$, $e \in \mathfrak{so}(\mathfrak{g})$ and $\alpha \in V \otimes \mathfrak{g}$. Then, $(b,a)$ exponentiates in $SO(V \oplus \mathfrak{g} \oplus V^*)$ to
\begin{equation}
\label{eq:Liealg}
e^{(b,a)} = \left (\begin{array}{ccc}
1 & 0 & 0 \\
a & 1 & 0\\
b - c(a \otimes a) & - 2c(a,\cdot) & 1
\end{array}\right ).
\end{equation}
and a straightforward calculation shows that $e^{b,a} = \Psi_\lambda^{-1}\Psi_{\lambda_{a,B}}$. Following \cite{Rubio}, such a symmetry will be called a $(b,a)$-\emph{transform}. Exponentiation endows the product $\Omega^2 \times \Omega^1(\mathfrak{g})$ with a natural group operation
$$
(b,a)\cdot(b',a') = (b + b' + c(a\wedge a'),a+a')
$$
that acts transitively on the space of isotropic splittings, where the right hand side corresponds to the element $e^{(b,a)} \cdot e^{(b',a')}$.

\subsection{Admissible generalized metrics}

Let us denote by $(t,s)$ the signature of the metric on $E$. A generalized metric of signature $(p,q)$, or simply, a metric on $E$ is a reduction of the $O(t,s)$-bundle of frames of $E$ to
$$
O(p,q)\times O(t-q,s-p) \subset O(t,s).
$$
Alternatively, it is given by a subbundle
$$
V_+ \subset E
$$
such that the restriction of the metric on $E$ to $V_+$ is a non-degenerate metric of signature $(p,q)$. We denote by $V_-$ the orthogonal complement of $V_+$ on $E$. A generalised metric determines a vector bundle isomorphism
$$
G\colon E \to E,
$$
with $\pm$ eigenspace $V_\pm$, which is symmetric $G^* = G$ and squares to the identity $G^2 = \Id$. The endomorphism $G$ determines completely the metric, as $V_+$ is recovered by
$$
V_+ = \Ker(G - \Id).
$$

\begin{definition}\label{def:admet}
A metric $V_+$ of arbitrary signature is \emph{admissible} if
$$
V_+ \cap T^* = \{0\} \qquad  \textrm{and} \qquad \rk V_+ = \rk E - \dim M.
$$
\end{definition}

The main difference with the definite case studied in \cite{G1} is that, in arbitrary signature, $V_+$ may cut $T^*$ on a non-trivial isotropic subspace. This motivates the first condition in the previous definition.

The second condition will imply that admissible metrics have a similar structure as definite generalized metrics on an exact Courant algebroid, that is, they are given by a pair consisting of an isotropic splitting of $E$ and a metric on $M$. To see this, note that for any admissible metric $V_+$ the restriction of $\pi_Q$ to $V_+$ induces an isomorphism with $Q = E/T^*$ by transversality. This isomorphism provides an isometry
$$
V_+ \cap \Ker \pi \cong \ad P
$$
and hence the orthogonal complement $(V_+ \cap \Ker \pi)^\perp \subset V_+$ inherits a non-degenerate metric $g$, that we can identify with a metric on $T$ via the isomorphism $\pi_\perp \colon (V_+ \cap \Ker \pi)^\perp \to T$ induced by $\pi$. The isotropic splitting determined by $V_+$ is then given by
$$
\lambda(X) = \pi_\perp^{-1}(X) - g(X)
$$
and, in summary, we have the following:

\begin{proposition}\label{prop:admissmetric}
An admissible metric $V_+$ is equivalent to a pair $(g,\lambda)$, where $g$ is a metric on $T$ and $\lambda \colon T \to E$ is an isotropic splitting such that
\begin{equation}\label{eq:V+exact}
V_+ = \{\lambda(X) + g(X) + \pi_{Q|V_+}^{-1}(r) \colon X \in T, r \in \ad P\}.
\end{equation}
\end{proposition}

Note that using the isomorphism $\Psi_\lambda$ \eqref{eq:Psilambda} the admissible metric has a very simple description
\begin{equation}\label{eq:V+exact}
V_+ = \{X + g(X) + r \colon X \in T, r \in \ad P\}.
\end{equation}
Definition \ref{def:admet} and Proposition \ref{prop:admissmetric} extend previous constructions by R. Rubio for abelian group $G$ \cite{Rubio2}.

%\subsubsection{Structure of admissible metrics}
Our next task is to describe an admissible metric in more geometric terms: we will see that any such metric $V_+$ determines a connection $A$ on the principal bundle $P$ and a $3$-form $H$ on $M$, related by \eqref{eq:dH}, which enable to describe the bracket in the splitting provided by $V_+$.

\begin{remark}
Geometrically, an admissible metric contains the information of an invariant metric on the total space of $P$ that induces the metric $c$ on the vertical bundle, that is, a metric on the base $M$ and a connection on $P$.
\end{remark}

\begin{proposition}\label{prop:admissmetric2}
An admissible metric $V_+$ on $E$ determines a metric $g$ and a $3$-form $H$ on $M$ and a connection $A$ on $P$, related by \eqref{eq:dH}, such that the bracket in the splitting provided by $V_+$ is given by \eqref{eq:bracket}. Conversely, any pair $(g,A)$ determines an admissible metric, uniquely up to transformation by $b \in \Omega^2$, given by
\begin{equation}\label{eq:V+exact}
V_+ = \{X + g(X) + r \colon X \in T, r \in \ad P\}
\end{equation}
in a splitting \eqref{eq:splittingA} provided by the connection $A$.
\end{proposition}

\begin{proof}
Let $A'$ be an auxiliary connection on $P$ and consider an isomorphism
$$
E \cong T \oplus \ad P \oplus T^*
$$
induced by $A'$, as described in Proposition \ref{prop:reducedcourant}, with $3$-form $H' \in \Omega^3$  and bracket \eqref{eq:bracket}. As the group $\Omega^2 \times \Omega^1(\ad P)$ acts transitively on the space of isotropic splittings, by Proposition \ref{prop:admissmetric} there exists a $(b,a)$-transform such that
$$
V_+ = e^{(b,a)}\{X + g(X) + r \colon X \in T, r \in \ad P\}.
$$
The first part of the statement reduces to calculate the Dorfman bracket in the splitting provided by $V_+$, given by
\begin{equation}\label{eq:bracketV+}
e^{-(b,a)}[e^{(b,a)}s_1,e^{(b,a)}s_2] \quad \textrm{for} \quad s_1,s_2 \in \Gamma(T \oplus \ad P \oplus T^*).
\end{equation}
To calculate this, we consider the isomorphism
$$
L_{A'} \colon T \oplus \ad P \oplus T^* \to K^\perp/G \colon X + r + \xi \mapsto A'^\perp X + r + c(r,A'\cdot) + p^*\xi,
$$
similarly as in the proof of Proposition \ref{prop:reducedcourant}. Then, defining $A = A' -a$ we notice that
$$
L_{A'}e^{(b,a)} = e^{p^*b + c(a \wedge A)}L_{A},
$$
where $e^{p^*b + c(a \wedge A)}$ denotes the $B$-transform by $p^*b + c(a \wedge A) \in \Omega^2(P)^G$. Therefore, arguing as in Proposition \ref{prop:reducedcourant}, we conclude that the bracket \eqref{eq:bracketV+} is given by \eqref{eq:bracket} with connection $A$ and $3$-form
$$
H = H' + db + d(c(a\wedge A)) - CS(A') + CS(A) \in \Omega^3.
$$
The converse is straightforward from Proposition \ref{prop:admissmetric}.
\end{proof}

\subsection{Generalized connections and torsion}\label{subsec:connectionstorsion}

A generalized connection $D$ \cite{G3} on $E$ is a first order differential operator
$$
D \colon \Gamma(E) \to \Gamma(E^* \otimes E)
$$
satisfying the Leibniz rule $D_e(fs) = fD_es + \pi(e)(f)s$, for $s,e \in \Gamma(E)$ and $f \in C^\infty(M)$. We will only consider generalized connections compatible with the inner product on $E$, that is, satisfying
$$
\pi(e)(\langle s_1,s_2 \rangle) = \langle D_e s_1,s_2 \rangle + \langle s_1,D_e s_2 \rangle.
$$

Given a (standard) connection $\nabla^E$ on $E$ compatible with $\langle\cdot,\cdot\rangle$ we can construct a generalized connection by
$$
D_e s = \nabla^E_{\pi(e)}s
$$
and we note that any other generalized connection differs from $D$ by an element
$$
\chi \in \Gamma(E^* \otimes \mathfrak{o}(E)).
$$
In particular, given a connection $A$ on $P$ and a connection $\nabla^T$ on $T$, we can consider the connection on $E$ induced by
\begin{equation}\label{eq:nablaE}
\nabla^E = \nabla^T \oplus \nabla^A \oplus \nabla^{T^*}
\end{equation}
in a splitting $E \cong T \oplus \ad P \oplus T^*$ provided by $A$ (see Proposition \ref{prop:reducedcourant} and \eqref{eq:splittingA}).

Let us fix a connection $A$. We now calculate an explicit formula for the generalized torsion of an arbitrary generalized connection
$$
D = \nabla^E + \chi,
$$
for $\nabla^E$ as in \eqref{eq:nablaE}. Recall that the generalized torsion, defined by Gualtieri in \cite{G3}, is an element $T_D \in \Lambda^3 E^*$ defined by
$$
T_D(e_1,e_2,e_3) = \langle D_{e_1}e_2 - D_{e_2}e_1 - [[e_1,e_2]],e_3 \rangle + \langle D_{e_3} e_1,e_2 \rangle,
$$
where
$$
[[e_1,e_2]] = \frac{1}{2}([e_1,e_2] - [e_2,e_1])
$$
is the skew symmetrization of the Dorfman bracket on $E$. Let us denote by $T_{\nabla}$ the torsion of the connection $\nabla^T$ and consider the natural projection (see Proposition \ref{prop:reducedcourant})
$$
\pi_Q \colon E\to TP/G.
$$
Denoting by $\lambda\colon T \to T \oplus \ad P \oplus T^*$ the canonical isotropic splitting and
\begin{align*}
\Sigma T_\nabla (e_1,e_2,e_3) & = \langle \pi^*\lambda(T_\nabla)(e_1,e_2),e_3\rangle -  \langle \pi^*\lambda(T_\nabla)(e_1,e_3),e_2\rangle\\
& + \langle \pi^*\lambda(T_\nabla)(e_2,e_3),e_1\rangle,\\
\Sigma\chi (e_1,e_2,e_3) & = \langle \chi_{e_1}e_2,e_3\rangle -  \langle \chi_{e_2}e_1,e_3\rangle + \langle \chi_{e_3}e_1,e_2\rangle,
\end{align*}
a straightforward calculation using \eqref{eq:bracket} lead us to the following formula.
\begin{lemma}
\begin{equation}\label{eq:torsion}
T_D = - \frac{1}{2}\pi^*H + \pi_Q^*(CS(A)) + \Sigma T_\nabla + \Sigma\chi.
\end{equation}
\end{lemma}

%\subsubsection{A torsion free metric connection}
Let $V_+$ be an admissible generalized metric on $E$, with associated endomorphism $G \colon E \to E$. We use now formula \eqref{eq:torsion} to define a canonical torsion-free generalized connection $D^0$ compatible with $G$, that is, satisfying
\begin{equation}\label{eq:LeviCivitacondition}
D G = 0, \qquad T_{D} = 0.
\end{equation}
Equivalently, a generalized connection is compatible with $G$ if and only if it preserves the factors in the orthogonal decomposition $E = V_+ \oplus V_-$.

In the sequel we use the splitting of $E$ given by $V_+$ (see Proposition \ref{prop:admissmetric})
$$
E \cong T \oplus \ad P \oplus T^*
$$
and consider the corresponding metric $g$ and connection $A$ (see Proposition \ref{prop:admissmetric2}). In this splitting $G$ has the following simple expression
\begin{equation*}%\label{eq:G}
G = \left(
\begin{array}{ccc}
0 & 0 & g^{-1}\\
0 & \Id & 0\\
g & 0 & 0
\end{array}\right).
\end{equation*}
Define the generalized connection
$$
D' = \nabla^g \oplus \nabla^A \oplus \nabla^{g^*}
$$
on $E$, induced by the Levi-Civita connection $\nabla^g$ of $g$. It is easy to see that $D'$ preserves $G$ and any other compatible generalized connection differs from $D'$ by
$$
\chi \in E^* \otimes \(\mathfrak{o}(V_+)\oplus \mathfrak{o}(V_-)\).
$$
By \eqref{eq:torsion} we have that the torsion of $D' + \chi$ equals
\begin{equation}\label{eq:torsionvariation}
T_{D' + \chi} = - \frac{1}{2}\pi^*H + \pi_Q^*(CS(A)) + \Sigma \chi.
\end{equation}
We want to find $\chi$ as before such that \eqref{eq:torsionvariation} vanishes. For this, we define $\chi^0 \in E^* \otimes \mathfrak{o}(E)$ by
\begin{equation}\label{eq:Gtransverse}
\chi^0_e = \left(
\begin{array}{ccc}
0 & 0 & 0\\
- i_X F & c^{-1}(c(r,[\cdot,\cdot])) & 0\\
i_X H - 2c(F,r) & 2c(i_X F,\cdot) & 0
\end{array}\right) \in \mathfrak{o}(E),
\end{equation}
for $e = X + r + \xi$. By construction, we have that
$$
\langle \chi^0_{e_1}e_2,e_3\rangle = \(\frac{1}{2}\pi^*H - \pi_Q^*(CS(A))\)(e_1,e_2,e_3)
$$
We introduce the notation
$$
\chi^{\pm\pm\pm}_{e_1}e_2 = \Pi_{\pm} \chi^0_{e_1^{\pm}}(e_2^{\pm}),
$$
where
$$
\Pi_{\pm} = \frac{1}{2}(\Id \pm G) \colon E \to V_{\pm}
$$
denote the orthogonal projections and $e_j^\pm = \Pi_\pm e_j$.
\begin{proposition}
The following expression defines a torsion free generalized connection compatible with $V_+$
\begin{equation}\label{eq:Levicivita}
D^0 = D' + \frac{1}{3} \chi^{+++} + \frac{1}{3} \chi^{---} + \chi^{+-+} + \chi^{-+-}
\end{equation}
\end{proposition}
The proof follows from formula \eqref{eq:torsionvariation} and a straightforward calculation and is left to the reader.

A remarkable fact about the connection $D^0$ is that it singles out four (standard) connections on $T$ with skew torsion, compatible with the metric $g$, given by
\begin{align*}
\nabla^\pm &= \nabla^g \pm \frac{1}{2}g^{-1}H,\\
\nabla^{\pm1/3} &= \nabla^g \pm \frac{1}{6}g^{-1}H.
\end{align*}
For our purpouses, we will focus in the pair given by $\nabla^{+}$ and $\nabla^{1/3}$. This pair of connections was used by Bismut \cite{Bismut} to prove a Lichnerowicz type formula for the square of the Dirac operator of $\nabla^{1/3}$, to which we will come back later (for physical applications see \cite{KiYi}).

To see the role of $\nabla^{+}$ and $\nabla^{1/3}$, we need to give explicit formulae for our connection $D^0$. Note that in the splitting $E \cong T \oplus \ad P \oplus T^*$ provided by $V_+$ we can write (see Proposition \ref{prop:admissmetric})
\begin{align*}
e_1^+ & = X + gX + r,\\
e_2^- & = Y - gY,\\
e_3^+ & = Z + gZ + t,
\end{align*}
for $e_j^+ \in V_+$, $j = 1,3$, and $e_2^- \in V_-$. Then, we have
\begin{equation}\label{eq:D-+}
\begin{split}
D^0_{e_2^-}e_3^+ %& = 2\Pi_+(\nabla^g_YZ) + \nabla_Yt + \chi^{+-+}_{e_2^-}e_3^+\\
%& = 2\Pi_+\(\nabla^g_YZ + \frac{1}{2}g^{-1}H(Y,Z,\cdot) + g^{-1}c(F(Y,\cdot),t)\)\\
%& + \nabla_Yt - F(Y,Z)\\
& = 2\Pi_+\(\nabla^+_YZ + g^{-1}c(F(Y,\cdot),t)\) + \nabla_Yt - F(Y,Z)\\
D^0_{e_1^+}e_3^+ %& = 2\Pi_+(\nabla^g_XZ) + \nabla_Xt + \frac{1}{3} \chi^{+++}_{e_1^+}e_3^+\\
& = 2\Pi_+\(\nabla^{1/3}_XZ + \frac{1}{3} g^{-1}c(F(X,\cdot),t) - \frac{1}{3} g^{-1}c(F(Z,\cdot),r)\)\\
& + \nabla_Xt - \frac{1}{3}F(X,Z) + \frac{1}{3}c^{-1}\(c(r,[t,\cdot])\)
%& + \frac{1}{3}\Pi_+\(\varphi(Z)e_1^+ - 2(g(X,Z) + c(r,t))\varphi\)
\end{split}
\end{equation}

\begin{remark}
The pair $\nabla^{-}$ and $\nabla^{1/3}$ was used by Bismut \cite{Bismut} to prove a local index  theorem for connections with skew torsion, and more recently by Ferreira \cite{Ferreira} to prove a vanishing theorem for twisted De Rham cohomology.
\end{remark}

%\subsubsection{The Weyl term}\label{subsubsec:Weylterm}

As pointed out in \cite{CSW}, the conditions \eqref{eq:LeviCivitacondition} do not determine $D$ uniquely. Indeed, any other torsion free generalized connection compatible with $V_+$ is of the form $D^0 + \chi^+ \oplus \chi^-$ for
\begin{equation}\label{eq:torfreecongen}
\chi^\pm \in V_\pm^* \otimes \mathfrak{o}(V_\pm) \quad  \textrm{such that}\quad \Sigma \chi^{\pm} = 0.
\end{equation}
We use now this fact to construct a family of such generalized connections parameterised by $\Omega^1(M)$, that we will use in our applications in \secref{sec:sugra}. Given a $1$-form $\varphi \in \Omega^1(M)$, regard $\varphi \in E^*$ %using the anchor
and consider the \emph{Weyl term}
$$
\chi^\varphi \in E^* \otimes\mathfrak{o}(E)
$$
defined by
\begin{equation}\label{eq:Weylterm}
\chi^\varphi_e e' = \varphi(e') e - \langle e,e' \rangle \langle\cdot,\cdot\rangle^{-1}\varphi.
\end{equation}
Changing now $\chi^0$ by $\chi^0 + \chi^\varphi$ in the definition \eqref{eq:Levicivita}, we obtain a new torsion-free, compatible, generalized connection $D^\varphi$.

To give explicit formulae for $D^\varphi$, we use again the splitting $E \cong T \oplus \ad P \oplus T^*$ provided by $V_+$. Then, using the same notation as in formulae \eqref{eq:D-+},
\begin{equation}\label{eq:Dvarphi}
\begin{split}
D^\varphi_{e_2^-}e_3^+ &= D^0_{e_2^-}e_3^+,\\
D^\varphi_{e_1^+}e_3^+ & = D^0_{e_1^+}e_3^+ + \frac{1}{3}\Pi_+\(\varphi(Z)e_1^+ - 2(g(X,Z) + c(r,t))\varphi\).
\end{split}
\end{equation}
More generally, we can take an arbitrary element $\varphi \in E^*$ to define $D^\varphi$.

\begin{remark}\label{rem:LeviCivita}
Let $\Aut E$ denote the group of automorphisms of $E$, that is, automorphisms of the vector bundle $E$ covering a diffeomorphism on $M$ and preserving the bracket and the product (see \cite[Definition 2.2]{BuCaGu}). Note that the map $V_+ \to D_0$ defined by assigning the canonical generalized connection $D^0$ to an admissible metric $V_+$ is natural, in the sense that it is equivariant with respect to the natural action of $\Aut E$ in the domain and the target. We believe that this condition is strong enough to determine such a map uniquely. If so, $D^0$ should be considered as \emph{the} analogue of the Levi-Civita in generalized geometry.
\end{remark}

\begin{remark}
The generalized connection $D^\varphi$ %constructed in Example \ref{example:Dphi}
provides an analogue of the Levi--Civita connection on a Weyl manifold. Recall that a pair $(g,\varphi)$ on $M$, given by a metric $g$ and $\varphi \in \Omega^1(M)$, determine a Weyl structure on $M$ \cite{Folland}. A Weyl structure carries a unique Levi--Civita connection $\nabla^\varphi$, that is, a torsion free and compatible connection, in the sense that
$$
\nabla^\varphi g + \varphi \otimes g = 0.
$$
Using the Levi-Civita connection $\nabla^g$ of $g$, we can write an explicit expression for $\nabla^\varphi$ given by
\begin{equation}\label{eq:Weylconnection}
\nabla = \nabla^g + \frac{1}{2} \varphi \otimes \Id + \chi^\varphi,
\end{equation}
where $\chi^\varphi \in \Omega^1(\mathfrak{o}(TM))$ is defined by
$$
\chi^\varphi_X Y = \frac{1}{2}(\varphi(Y)X - g(X,Y)g^{-1}\varphi).
$$
To see the analogy, note that the second summand in \eqref{eq:Weylconnection} does not lie in $\Omega^1(\mathfrak{o}(TM))$ and hence does not have a natural analogue in generalized geometry (it cannot be expressed in terms of the principal bundle of orthonormal frames of $E$). Therefore, we see that our expression for $D^\varphi$ is formally as \eqref{eq:Weylconnection} once we drop $\varphi \otimes \Id/2$.
\end{remark}

\section{Generalized curvature}\label{sec:Ricci}

In this section we calculate the generalized curvature and the generalized Ricci tensor of an admissible metric on the reduced Courant algebroid $E$, with respect to an arbitrary torsion free, compatible generalized connection. In particular, we give explicit formulae for the family of generalized connections $D^\varphi$ constructed in Section \ref{subsec:connectionstorsion}.

\subsection{Generalized curvature}

The generalized curvature of a generalized connection $D$ \cite{G3} is defined by
$$
GR(e_1,e_2) = D_{e_1}D_{e_2} - D_{e_2}D_{e_1} - D_{[[e_1,e_2]]} \in \mathfrak{o}(E)
$$
for $e_1,e_2 \in C^\infty(E)$. This quantity becomes tensorial evaluated on a pair of orthogonal sections. In particular, given a generalized metric
$$
E = V_+ \oplus V_-
$$
we obtain a tensor by restriction
$$
GR \in V_+^* \otimes V_-^* \otimes \mathfrak{o}(E).
$$

Let us fix $\varphi \in \Omega^1(M)$ and an admissible metric $V_+ \subset E$, with corresponding metric $g$ and connection $A$. Consider the torsion free, compatible, generalized connection $D^\varphi$ constructed in Section \ref{subsec:connectionstorsion}. %, that we denote simply by $D$.
We want to calculate explicitely its curvature
$$
GR(e_1^+,e_2^-)e_3^+ \in V_+,
$$
for $e_j^+ \in V_+$, $j = 1,3$, and $e_2^- \in V_-$. To find a closed formula for $GR$, we define a tensorial quantity depending on the pair of connections $\nabla^+$ and $\nabla^{1/3}$, that is,
\begin{equation}\label{eq:R3}
\begin{split}
R^{1/3}(X,Y)Z & = \nabla^{1/3}_X\nabla^+_YZ - \nabla^+_Y\nabla^{1/3}_XZ - \(\frac{\nabla^{1/3}+\nabla^+}{2}\)_{[X,Y]}Z\\
& + i_{\nabla^{1/3}_XY + \nabla^+_YX - 2/3g^{-1}H(X,Y,\cdot)}\(\frac{\nabla^{1/3}-\nabla^+}{2}\)Z.
\end{split}
\end{equation}
The tensor $R^{1/3}$ is an hybrid of the curvatures of $\nabla^+$ and $\nabla^{1/3}$ (cf. \cite[Th. 1.1]{Ferreira2}), that can be written more explicitely as
\begin{align*}
R^{1/3}(X,Y)Z & = R^g(X,Y)Z + g^{-1}\Bigg{(}\frac{1}{2}(\nabla^g_XH)(Y,Z,\cdot) - \frac{1}{6}(\nabla^g_YH)(X,Z,\cdot)\\
& + \frac{1}{12}H(X,g^{-1}H(Y,Z,\cdot),\cdot) - \frac{1}{12}H(Y,g^{-1}H(X,Z,\cdot),\cdot)\\
& - \frac{1}{6}H(Z,g^{-1}H(X,Y,\cdot),\cdot)\Bigg{)}.
\end{align*}

For the calculations below, we set $\nabla = \nabla^A$ and denote $D^\varphi$ by $D$, to simplify the notation. Using the splitting $E \cong T \oplus \ad P \oplus T^*$ provided by $V_+$ (see Proposition \ref{prop:admissmetric}) and the notation in \eqref{eq:D-+} we obtain
\begin{align*}
D_{e_1^+}D_{e_2^-}e_3^+ & = 2\Pi_+\Bigg{(}\nabla_X^{1/3}\nabla^+_YZ + \nabla_X^{1/3}(g^{-1}c(F(Y,\cdot),t)) - \frac{1}{3}g^{-1}c(F(\nabla^+_YZ,\cdot),r)\\
& + \frac{1}{3}g^{-1}c(F(X,\cdot),\nabla_Yt)\\
& - \frac{1}{3}g^{-1}c(F(g^{-1}c(F(Y,\cdot),t),\cdot),r) - \frac{1}{3} g^{-1}c(F(X,\cdot),F(Y,Z))\\
& - \frac{1}{3}g(X,\nabla^+_YZ + g^{-1}c(F(Y,\cdot),t))\varphi\\
& - \frac{1}{3}c(r,\nabla_Yt - F(Y,Z))\varphi\Bigg{)}\\
& + \frac{1}{3}\varphi(\nabla^+_YZ + g^{-1}c(F(Y,\cdot),t))e_1^+\\
& +  \nabla_X\nabla_Yt -  \nabla_X(F(Y,Z)) + \frac{1}{3}c^{-1}\(c(r,[\nabla_Yt,\cdot])\) - \frac{1}{3}F(X,\nabla^g_YZ)\\
& - \frac{1}{3}c^{-1}\(c(r,[F(Y,Z),\cdot])\)\\
& - \frac{1}{6}F(X,g^{-1}H(Y,Z,\cdot)) - \frac{1}{3}F(X, g^{-1}c(F(Y,\cdot),t)),
\end{align*}
\begin{align*}
D_{e_2^-}D_{e_1^+}e_3^+ & = 2\Pi_+\Bigg{(}\nabla_Y^+\nabla^{1/3}_XZ + \frac{1}{3}\nabla_Y^+(g^{-1}c(F(X,\cdot),t)) - \frac{1}{3}\nabla_Y^+(g^{-1}c(F(Z,\cdot),r))\\
& + g^{-1}c(F(Y,\cdot),\nabla_Xt)\\
& - \frac{1}{3}g^{-1}c(F(Y,\cdot),F(X,Z)) + \frac{1}{3} g^{-1}c(r,[t,F(Y,\cdot)])\\
& + \frac{1}{3}\varphi(Z)\(\nabla^+_YX + g^{-1}c(F(Y,\cdot),r)\)\\
& - \frac{1}{3}(g(X,Z) + c(r,t))\nabla^+_Y\varphi\Bigg{)}\\
& + \frac{1}{3}i_Yd(\varphi(Z))e_1^+ - \frac{1}{3}i_Yd(g(X,Z + c(r,t)))2\Pi_+\varphi\\
& +  \nabla_Y\nabla_Xt - \frac{1}{3} \nabla_Y(F(X,Z)) + \frac{1}{3} \nabla_Y(c^{-1}\(c(r,[t,\cdot])) \) - F(Y,\nabla^g_XZ)\\
& - \frac{1}{6}F(Y,g^{-1}H(X,Z,\cdot))  + \frac{1}{3}F(Y,g^{-1}c(F(Z,\cdot),r)) - \frac{1}{3}F(Y,g^{-1}c(F(X,\cdot),t))\\
& + \frac{1}{3}\varphi(Z)(\nabla_Yr - F(Y,X)) + \frac{1}{3}(g(X,Z) + c(r,t))F(Y,g^{-1}\varphi).
\end{align*}
Using the equality
\begin{align*}
[[e_1^+,e_2^-]] & = [X,Y] - g(\nabla^g_XY + \nabla^g_YX,\cdot) + H(X,Y,\cdot)\\
& - F(X,Y) - \nabla_Y r - 2c(F(Y,\cdot),r)
\end{align*}
we also obtain
\begin{align*}
D_{[[e_1^+,e_2^-]]}e_3^+ & =  2\Pi_+\(\nabla^g_{[X,Y]}Z\) + \nabla_{[X,Y]}t + \frac{1}{3}\chi^{+++}_{[[e_1^+,e_2^-]]}e^+_3 + \chi^{+-+}_{[[e_1^+,e_2^-]]}e^+_3\\
& = 2\Pi_+\Bigg{(}\nabla^g_{[X,Y]}Z + \frac{1}{6}g^{-1}H(2[X,Y] + \nabla^g_XY + \nabla^g_YX,Z,\cdot)\\
&  - \frac{1}{6}g^{-1}H(g^{-1}H(X,Y,\cdot),Z,\cdot)\\
& + \frac{1}{3}g^{-1}c(F(Z,\cdot),\nabla_Yr)) + \frac{1}{3}g^{-1}c(F(2[X,Y] + \nabla^g_XY + \nabla^g_YX,\cdot),t)\\
& + \frac{1}{3}g^{-1}H(g^{-1}c(F(Y,\cdot),r),Z,\cdot) + \frac{1}{3}g^{-1}c(F(Z,\cdot),F(X,Y))\\
& - \frac{1}{3}g^{-1}c(F(g^{-1}H(X,Y,\cdot),\cdot),t) + \frac{2}{3}g^{-1}c(F(g^{-1}c(F(Y,\cdot),r),\cdot),t)\\
& + \frac{1}{6}\varphi(Z)\([X,Y] - \nabla_Y^gX - \nabla^g_XY + g^{-1}H(X,Y,\cdot) - 2g^{-1}c(F(Y,\cdot),r)\)\\
& - \frac{1}{6}(g(Z,[X,Y] - \nabla^g_YX - \nabla^g_XY) + H(X,Y,Z))\varphi\\
& + \frac{1}{3}\(c(F(Y,Z),r) + c(t,F(X,Y) + \nabla_Yr)\)\varphi\Bigg{)}\\
& + \nabla_{[X,Y]}t - \frac{1}{3}c^{-1}\(c(F(X,Y) + \nabla_Yr,[t,\cdot])\) - \frac{1}{3}\varphi(Z)(F(X,Y) - \nabla_Yr)\\
& + \frac{1}{3}F(Z,2[X,Y] + \nabla^g_XY + \nabla^g_YX - g^{-1}H(X,Y,\cdot) + 2g^{-1}c(F(Y,\cdot),r)).
\end{align*}
Finally, the identities
\begin{align*}
R_\nabla(X,Y)t & = [F(X,Y),t]\\
\nabla_Y(c^{-1}(c(r,[t,\cdot]))) & = c^{-1}(c(r,[\nabla_Y t,\cdot])) - c^{-1}(c(\nabla_Yr,[t,\cdot]))\\
(\nabla_X^{g}(c(F,t)))(Y,Z) & =  i_Z(\nabla_X^g(c(F(Y,\cdot),t))) - c(F(\nabla^g_XY,Z),t)\\
(\nabla_X^{g}H)(Y,Z,\cdot) & =  \nabla_X^g(H(Y,Z,\cdot)) - H(\nabla^g_XY,Z,\cdot) - H(Y,\nabla^gZ,\cdot)
\end{align*}
lead us to the following expression for the curvature
\begin{lemma}\label{lem:curvature}
\begin{align*}
GR(e_1^+,e_2^-)e_3^+ & = 2\Pi_+\Bigg{(}R^{1/3}(X,Y)Z + g^{-1}\Bigg{(}i_Y\(\nabla^{1/3}_X(c(F,t)) - c(F,\nabla_Xt)\)\\
& + \frac{1}{3}i_Z\(\nabla^{+}_Y(c(F,r)) - c(F,\nabla_Yr)\) - \frac{1}{3}i_X\(\nabla^{+}_Y(c(F,t)) - c(F,\nabla_Yt)\)\\
& + \frac{1}{3}H(Z,g^{-1}c(F(Y,\cdot),r),\cdot) + \frac{2}{3}c(F(g^{-1}H(X,Y,\cdot),\cdot),t)\\
& - \frac{1}{3}c(F(g^{-1}c(F(Y,\cdot),t),\cdot),r) - \frac{2}{3}c(F(g^{-1}c(F(Y,\cdot),r),\cdot),t)\\
& - \frac{1}{3}c(F(X,\cdot),F(Y,Z)) - \frac{1}{3}c(F(Z,\cdot),F(X,Y))\\
& + \frac{1}{3}c(F(Y,\cdot),F(X,Z)) - \frac{1}{3}c(r,[t,F(Y,\cdot)])\\
& + \frac{1}{3}(g(X,Z) + c(r,t))\nabla^{+}_Y\varphi\Bigg{)}\Bigg{)}\\
& - \frac{1}{3}\(i_Z\(\nabla^+_Y\varphi\) - c(F(Y,g^{-1}\varphi),t)\)e_1^+\\
& + \frac{4}{3}[F(X,Y),t] + \frac{1}{3}[F(Y,Z),r]\\
& + \frac{1}{3}\Big{(}\nabla_Y(F(X,Z)) - F(\nabla_Y^gX,Z) - F(X,\nabla^g_YZ)\Big{)}\\
& - \Big{(}\nabla_X(F(Y,Z)) - F(\nabla_X^gY,Z) - F(Y,\nabla^g_XZ)\Big{)}\\
& - \frac{1}{6}F(X,g^{-1}H(Y,Z),\cdot) + \frac{1}{6}F(Y,g^{-1}H(X,Z),\cdot)\\
& + \frac{1}{3}F(Z,g^{-1}H(X,Y),\cdot)\\
& - \frac{1}{3}F(X,g^{-1}c(F(Y,\cdot),t)) - \frac{1}{3}F(Y,g^{-1}c(F(Z,\cdot),r))\\
& + \frac{1}{3}F(Y,g^{-1}c(F(X,\cdot),t)) - \frac{2}{3}F(Z,g^{-1}c(F(Y,\cdot),r))\\
& - \frac{1}{3}(g(X,Z) + c(r,t))F(Y,g^{-1}\varphi)
\end{align*}
\end{lemma}
A careful inspection shows that this expression is tensorial in $e_1^+,e_2^-$. If we use instead the torsion free compatible connection $D + \chi^+ \oplus \chi^-$, with $\chi^\pm$ as in \eqref{eq:torfreecongen}, the curvature receives an additional contribution of the form
$$
\chi^+_{e_1^+}(D_{e_2^-}e_3^+) - D_{e_2^-}(\chi^+_{e_1^+}e_3^+) - \chi^+_{[[e_1^+,e_2^-]]}e_3^+.
$$

\subsection{Generalized Ricci tensor}

The generalized Ricci tensor
$$
\operatorname{GRic} \in V_-^* \otimes V_+^*
$$
of the generalized metric $G$ with torsion free connection $D$ acting on $e_2^-, e_3^+$ is defined as the trace of the endomorphism
$$
%\operatorname{Ric}^G(e_2^-,e_3^+) = \tr\(
e_1^+ \to R(e_1^+,e_2^-)e_3^+.
$$
To calculate $\operatorname{GRic}$, define the tensors
\begin{align*}
H \circ H (Y,Z) & := \sum_{j=1}^n g(H(e_j,Y,\cdot),H(e_j,Z,\cdot)),\\
F \circ F (Y,Z) & := \sum_{j=1}^m c(r_j,F(Y,g^{-1}c(F(Z,\cdot),r_j))),
\end{align*}
where $\{e_j\}^n_{j=1}$ and $\{r_j\}^m_{j=1}$ are, respectively, orthonormal basis for $g$ and $c$. Using $H \circ H$, we can now write the Ricci tensor of the connection $\nabla^+$ as
$$
\operatorname{Ric}^+ = \operatorname{Ric}^g - \frac{1}{4} H \circ H - \frac{1}{2}d^*H,
$$
where $\operatorname{Ric}^g$ denotes the Ricci tensor of $g$. Recall that the adjoint of the exterior differential of a $k$-form $\beta$ can be calculated as
$$
d^*\beta = - \sum_{j=1}^n i_{e_j}(\nabla^{g*}_{e_j}\beta)
$$
and %hence
%\begin{align*}
%d^*H(Y,Z) & = - \sum_{j=1}^n(\nabla^{g*}_{e_j}H)(Y,Z,e_j),\\
%d^*(c(F,t))(Y) & = \sum_{j=1}^n(\nabla^{g*}_{e_j}(c(F,t)))(Y,e_j)\\
%& = c(d_A^*F,t) - \sum_{j=1}^nc(F,\nabla_{e_j}t)(e_j,Y).
%\end{align*}
%Finally, using the identity
and also that
$$
i_Y *(F \wedge *H) = -(-1)^{ind(g)+n-1}\sum_{j=1}^n F(e_j,g^{-1}H(e_j,Y,\cdot)),
$$
where $ind(g)$ denotes the index of the metric $g$. Then, as a straightforward consequence of Lemma \ref{lem:curvature}, we obtain the desired expression for the generalized Ricci tensor.
\begin{proposition}\label{prop:Ricci}
\begin{align*}
\operatorname{GRic}(e_2^-,e_3^+) %& = \(\operatorname{Ric}^g - \frac{1}{4} H \circ H - \frac{1}{2}d^*H - F\circ F\)(Y,Z)\\
%& + \sum_j \(\nabla^{g*}_{e_j}(c(F,t)) - c(F,\nabla_{e_j}t)\)(Y,e_j)\\
%& - \frac{1}{2} \sum_j c(F(e_j,g^{-1}H(e_j,Y,\cdot)),t)\\
& = \(\operatorname{Ric}^+ - F\circ F - \frac{\rk V_+-1}{3}\nabla^+\varphi\)(Y,Z)\\
& + i_Y c\(d_A^*F + \frac{(-1)^{ind(g)+n-1}}{2}*(F\wedge *H) - \frac{\rk V_+-1}{3}F(g^{-1}\varphi,\cdot),t\).
\end{align*}
\end{proposition}
%\begin{proof}
%From Lemma \ref{lem:curvature} it follows that
%\begin{align*}
%\operatorname{Ric}^G(e_2^-,e_3^+) & = \(\operatorname{Ric}^g - \frac{1}{4} H \circ H - F\circ F - %\frac{\rk V_+-1}{3}\nabla^+\varphi\)(Y,Z)\\
%& + \frac{1}{2}\sum_{j=1}^n(\nabla^{g*}_{e_j}H)(Y,Z,e_j)\\
%& + \sum_j \(\nabla^{g*}_{e_j}(c(F,t)) - c(F,\nabla_{e_j}t)\)(Y,e_j)\\
%& - \frac{1}{2} \sum_j c(F(e_j,g^{-1}H(e_j,Y,\cdot)),t)\\
%& - \frac{\rk V_+-1}{3}c(F(g^{-1}\varphi,Y),t)
%\end{align*}
%and hence the statement holds by the previous formulae.
%\end{proof}

%\begin{question}
%\textbf{Is this quantity independent of the additional terms $\chi^+ \oplus \chi^-$?} Nigel says yes ...
%\end{question}

%\begin{question}
%\textbf{What is the definition of the generalized scalar curvature?} It uses the cubic dirac operator and the laplacian (cf. %notes of lecture on generalized geometry by Nigel and the paper by Ivanov and Friedrich)
%\end{question}

\section{Application to Supergravity}\label{sec:sugra}

In this section we apply the theory of admissible metrics on a transitive Courant algebroid, developed in the previous sections, to provide a novel description of the field equations of heterotic supergravity. To clarify ideas an exhibit the main differences of our approach with that of \cite{CSW}, we first discuss an oversimplified version of Type II supergravity.

\subsection{Type II supergravity}\label{subsec:typeII}

Here we essentially follow the description of Type II supergravity in \cite{CSW}, focussing on the new aspects provided by our construction. This is a ten dimensional supergravity theory on spin manifold $M$, i.e. oriented, with $w_2(M) = 0$, and a choice of $\gamma \in H^1(X,\ZZ_2)$, which arises in the low-energy limit of Type II string Theory.

The (bosonic) field content of this theory is given by a metric $g$ of signature $(1,9)$, a dilaton $\phi \in C^\infty(M)$ and a $3$-form $H \in \Omega^3(M)$, with action
$$
\int_M e^{-2\phi}\(S^g + 4|d\phi|^2 - \frac{1}{2}|H|^2\)\Vol_g.
$$
The equations of motion of this theory are therefore given by (with the notation in \S\ref{sec:Ricci})
\begin{equation}\label{eq:motionTIIsugra}
\begin{split}
\operatorname{Ric}^g + 2 \nabla^g(d\phi) - \frac{1}{4} H \circ H & = 0,\\
S^g + 4 \Delta \phi - 4 |d\phi|^2 - \frac{1}{2}|H|^2 & = 0,\\
d^*(e^{-2\phi}H)& = 0,
\end{split}
\end{equation}
where $|H|^2 = \frac{1}{6}\sum_{i,j,k}H(e_i,e_j,e_k)^2$ with respect to an orthonormal frame for $g$. Crucially, consistency also requires the condition
$$
dH = 0.
$$
Note that we have set the RR fields to zero and neglected the fermionic equations of motion \cite{CSW} to simplify the discussion.

To describe the previous conditions in terms of generalised geometry, we fix a cohomology class $\Omega \in H^3(M,\RR)$. Consider the exact Courant algebroid
$$
0 \to T^*M \to E \to TM \to 0
$$
on $M$ determined by this class (see e.g. \cite{G1annals}). Recall that an admissible metric $V_+$ on $E$ (see Definition \ref{def:admet}) determines a pseudo-Riemannian metric $g$ on $M$ and a closed $3$-form $H \in \Omega$ and therefore, up to the dilaton field, an admissible metric of signature $(1,9)$ contains precisely the field content of Type II supergravity. We should notice that our defition of admissible metric \ref{def:admet} makes sense for exact Courant algebroids. In this context, an admissible metric is a generalised metric
$$
E = V_+ \oplus V_-
$$
of arbitrary signature that is transverse to $T^*M$.

To introduce the dilaton, we use the freedom in the choice of torsion-free generalized connection compatible with $V_+$ as explained in Section \ref{subsec:connectionstorsion}. Given a dilaton $\phi \in C^\infty(M)$, we consider the torsion free, $V_+$-compatible, generalized connection $D^\phi$ determined by the 1-form
$$
\varphi = -\frac{2}{3} d\phi.
$$
Explicitly, in the splitting $E \cong T \oplus T^*$ provided by $V_+$, we have
\begin{align*}
D_{v^\mp}^\phi w^\pm  & = 2\Pi_\pm (\nabla^\pm_XY)\\
D_{v^\pm}^\phi w^\pm  & = 2\Pi_\pm\(\nabla^{\pm 1/3}_XY + \frac{1}{3}\(\varphi(Y)X - 2g(X,Y)g^{-1}\varphi\)\)\\
& = 2\Pi_\pm\(\nabla^{\pm 1/3}_XY - \frac{2}{9}d\phi(Y)X + \frac{4}{9}g(X,Y)g^{-1}d\phi\),
\end{align*}
where $v^\mp = X \pm gX$ and $w^\pm = Y \pm gY$. This formula shall be compared with formula (4.31) in \cite{CSW}. Using the notation in Proposition \ref{prop:Ricci}, this lead us to the following formula for the generalized Ricci tensor of $D^\phi$, that provides a unified description of the first and third equations in \eqref{eq:motionTIIsugra}
\begin{align*}
\operatorname{GRic} & = \operatorname{Ric}^g + 2\nabla^g(d\phi) - \frac{1}{4} H \circ H - \frac{e^{2\phi}}{2}d^*(e^{-2\phi}H).
\end{align*}

For the second equation, we follow the interpretation in \cite{CSW} in terms of the \emph{generalized scalar curvature}. Consider the standard decomposition of the spinor bundle
$$
S(V_+) = S_+(V_+) \oplus S_-(V_+)
$$
The generalised scalar curvature $\operatorname{GS}$ of $D^\phi$ arises from a Lichnerowicz type formula that compares the square of the Dirac operator
$$
D\hspace{-0.2cm}\slash_{+}^{\;\phi}  \hspace{0.1cm} \colon S_+(V_+) \to S_-(V_+)
$$
of the induced operator $D^\phi_+\colon V_+ \to V_+^* \otimes V_+$ with the rough Laplacian of $D^\phi_-\colon V_+ \to V_-^* \otimes V_+$ acting on spinors. More concretely
$$
(D\hspace{-0.2cm}\slash_{+}^{\;\phi} )^2 \epsilon_+ = ((D^\phi_{-})^*D^\phi_{-} + \operatorname{GS})\epsilon_+
$$
for any section $\epsilon_+$ of $S_+(V_+)$, where (cf. \cite[p.21]{CSW})
$$
\operatorname{GS} = S^g + 4 \Delta \phi - 4 |d\phi|^2 - \frac{1}{2}|H|^2.
$$
As pointed out to me by N. Hitchin, taking into account the relation of the Generalised connection $D^\phi$ with the standard connections $\nabla^+$ and $\nabla^{1/3}$, this formula should be compared with Bismut's Lichnerowicz type formula for the cubic Dirac operator \cite{Bismut} (see also \cite[Th. 2.1]{Ferreira}).

The next result is a direct consecuence of the previous formulae.
\begin{theorem}\label{th:typeII}
An admissible metric $G$ on $E$ with torsion free, $G$-compatible, generalized connection $D^\phi$ satisfies
$$
\operatorname{GRic} = 0, \qquad \operatorname{GS} = 0,
$$
if and only if the corresponding tuple $(g,\phi,H)$ satisfies the field equations of Type II supergravity \eqref{eq:motionTIIsugra}.
\end{theorem}
In contraposition to \cite{CSW}, here the quatities $\operatorname{GRic}$ and $\operatorname{GS}$ depend on the choice of torsion free generalised connection. The special role played by the connection $D^0$ in generalised geometry certainly deserves further studies, that we leave for elsewhere (see Remark \ref{rem:LeviCivita}).

In addition to the equations of motion, in order to find a classical solution of the theory supersymmetry requires the existence of a pair of  %\textbf{Mayorana-Weyl??}
spinors $\epsilon_\pm \in S_\pm(V_+)$ satisfying the \emph{Killing spinor equations} \cite{CSW}.
% (\textbf{in \cite{FIVU} and Papadopoulos papers the dilatino equation is written with $-$. Strominger \cite{Strom} writes the first equation using $\nabla^-$ and the second with $+$ and a different factor, in eq. 2.5. Is this a matter of different conventions? Is there any relation between the chirality of the spinor and the `sign' of the connection? Should be any relation between the relative signs of the fist two equations?}) James says that part of it is a matter of convention ($H$ can flipped with $-H$) bit there is a relative sign that must be preserved (but he does not remember which one). We use the notation of Sparks-Martelli \cite{MaSp}
%\begin{equation}\label{eq:susyTIIsugra}
%\nabla^\pm \epsilon_\pm = 0, \qquad D\hspace{-0.3cm}\slash^{\;\pm 1/3} \epsilon = 0,
%\end{equation}
%where $D\hspace{-0.3cm}\slash^{\;\pm 1/3}$ denotes the Dirac operator of the connection $\nabla^{\pm 1/3}$.
As pointed out in \cite{CSW}, this equations admit an elegant expression in terms of the generalised connections $D^\phi$ (e.g. for $\epsilon_+$)
\begin{equation}\label{eq:susyTIIsugra}
D^\phi_- \epsilon_+ = 0, \qquad D\hspace{-0.2cm}\slash_{+}^{\;\phi}  \hspace{0.1cm} \epsilon_+ = 0,
\end{equation}

\subsection{Heterotic supergravity}

The low-energy limit of heterotic string theory, often called heterotic supergravity, is formulated on a $10$-dimensional spin manifold $M$ endowed with a principal bundle $P_{YM}$ with compact structure group $K$, contained in $SO(32)$ or $E_8 \times E_8$. We denote by $-\tr_\mathfrak{k}$ the killing form on the Lie algebra $\mathfrak{k}$ of $K$.

In addition to the (bosonic) field content of Type II supergravity, in heterotic supergravity there is also a gauge field $A$, that is, a connection on $P_{YM}$. The action is given by
$$
\int_M e^{-2\phi}\(S^g + 4|d\phi|^2 - \frac{1}{2}|H|^2 + \frac{\alpha'}{2}(|R|^2- |F_A|^2)) \)\Vol_g,
$$
where $\alpha'$ is a real (positive) parameter of the theory, the norm squared of the curvature $F_A$ is taken with respect to $-\tr_\mathfrak{k}$ (we use the notation in \cite{MaSp}, with the opposite sign convention for the killing form) and $R$ is the curvature of a spin connection $\nabla^T$ on $M$ for the metric $g$. Here, the norm of $R$ is taken with respect to the corresponding killing form on $\mathfrak{so}(1,9)$, that we denote $-\tr$.

An important (and particularly confusing) feature of this theory is that the $H$-field is locally described in terms of a $B$-field potential and the Chern--Simons $3$-forms of $A$ and $\nabla^T$ as
\begin{equation}\label{eq:Hsugra}
H = dB + \alpha'(CS(\nabla^T) - CS(A)).
\end{equation}
Although this expression can be apparently made rigorous using a complicated definition for the $B$-field \cite{Wi2000}, we take it here as a rather formal expression. It implies a well defined global constraint for the exterior differential of the $H$-field, namely
\begin{equation}\label{eq:dHsugra}
dH = \alpha'(\tr R \wedge R - \tr_\mathfrak{k} F_A \wedge F_A),
\end{equation}
that arises as a requirement to have an anomaly free theory.

Taking the local form \eqref{eq:Hsugra} of the $H$-field into account, the equations of motion of this theory can be written as (with the notation in \S\ref{sec:Ricci} and $c = \tr -\tr_\mathfrak{g}$)
\begin{equation}\label{eq:motionHetsugra}
\begin{split}
\operatorname{Ric}^g + 2 \nabla^g(d\phi) - \frac{1}{4} H \circ H + \alpha' \tr_{\mathfrak{k}}F_A \circ F_A - \alpha' \tr R \circ R & = 0,\\
%S^g + 4 \Delta \phi - 4 |d\phi|^2 - \frac{1}{12}|H|^2 + \alpha'(|R|^2 - |F_A|^2) & = 0,\\
d^*(e^{-2\phi}H) & = 0,\\
d^*_A(e^{-2\phi}F_A) + \frac{e^{-2\phi}}{2} *(F_A \wedge *H) & = 0,\\
d^*_{\nabla^T}(e^{-2\phi}R) + \frac{e^{-2\phi}}{2}*(R \wedge *H) & = 0,
\end{split}
\end{equation}
The second summand in the last two equations arises from the variation of the term $|H|^2$ and \eqref{eq:Hsugra}. The last equation comes from the variation of the connection $\nabla^T$, that we consider here as an additional field of the theory.

\begin{remark}
%The last equation has the net effect of turning $\nabla^T$ into a dynamical field. This, as well as the choice of connection $\nabla^T$ is a controverted issue and deserves a few comments.
In the physics literature (the unfamiliar reader can skip this remark), each field in heterotic supergravity is considered as a formal power series in $\alpha'$. In \cite{Hull2} it is argued that, when considering the full expansion of the fields and taking into account the supersymmetry variations, the connection $\nabla^T$ must equal the connection with skew-torsion $\nabla^-$ (with connection $\nabla^+$ for the supersymmetry variations \cite{MaSp}).
%In \cite{MaSp} it is recalled that, whatever the convention is, there must be always a relative %sign between the connection $\nabla^{\pm}$ in the gravitino equation and the connection used to %calculate $R$ (note that in \cite{Hull} it was argued that any spin connection is valid, but in %this reference only the Bianchi identity is taken into account).
In this work, however, we are going to consider the equations of motion up to one loop as exact. Taking into account the supersymmetry variations up to one loop, to obtain a sensible theory we must consider the connection $\nabla^T$ as an additional field in the theory (otherwise we obtain an overdetermined system in the light of the main result in \cite{Ivan09} (see also \cite{FIVU})).
\end{remark}

\begin{remark}
As is common in the literature (see e.g. \cite{FIVU}), we have neglected the equation coming from the variation of the dilaton
$$
S^g + 4 \Delta \phi - 4 |d\phi|^2 - \frac{1}{2}|H|^2 + \alpha'(|R|^2 - |F_A|^2) = 0.
$$
This is justified because this equation is a esentially a consequence of the others. More precisely, applying the covariant derivative $\nabla^g$ in the first equation of \eqref{eq:motionHetsugra} and contracting with the metric, combined with the second equation, we obtain that the left hand side is constant (to the knowledge of the author, this was originally observed in \cite{CFMP}. See also \cite[\S 6.9]{QFS}). This constant can be proved to be zero taking the supersymmetry variations into account. A proper understanding of the supersymetry variations and the dilaton equation in the framework of generalised geometry requires the study of spinors in transitive algebroids, that we hope to address in future work.
\end{remark}

We now apply our theory to provide a natural description of the field equations \eqref{eq:motionHetsugra} and \eqref{eq:dHsugra} in terms of generalized geometry. Let $P_M$ be the %double cover of the
$SL(10,\mathbb{R})$-bundle of oriented frames of the tangent bundle. Let $P$ be the principal bundle given by the cross product of $P_YM$ and $P_M$. On the Lie algebra of the structure group $\mathfrak{g} = \mathfrak{k} \oplus \mathfrak{sl}(10,\RR)$, we fix the non degenerate pairing
$$
c = \alpha'(\tr_\mathfrak{k} - \tr),
$$
given by re-scaling the difference of the killing forms on $\mathfrak{sl}(10,\RR)$ and $\mathfrak{k}$. Assume that the first Pontrjagin class of $P$ with respect to $c$ vanishes $p_1(P) = 0$ or, equivalently,
$$
p_1(P_{YM}) = p_1(P_M).
$$
Note that this is the necessary and sufficient condition to have solutions of the anomaly equation \eqref{eq:dHsugra}.  By Proposition \eqref{lemma:canonical}, this condition determines a canonical exact Courant algebroid
$$
0 \to T^*P \to \hat E \to TP \to 0
$$
endowed with a lifted $G$-action with non-degenerate pairing $c$ (see \eqref{eq:quadratic}) and such that it admits an equivariant isotropic splitting. Consider the reduced transitive Courant algebroid
$$
0 \to T^* \to E \to T \to 0
$$
given by Proposition \ref{prop:reducedcourant}.

On the reduced algebroid $E$, we consider admissible metrics $V_+$ such that $g$ has signature $(1,9)$ and the connection on $P$ is a product of a connection $A$ on $P_{het}$ and a spin connection $\nabla^T$ on $P_M$ compatible with $g$. With this assumption, we have
$$
dH = \alpha'(\tr R\wedge R - \tr F_A\wedge F_A)
$$
and therefore, up to the dilaton field, an admissible metric $V_+$ contains precisely the field content of heterotic supergravity. To introduce the dilaton, we use the freedom in the choice of torsion-free generalized connection compatible with $V_+$ as explained in Section explained in Section \ref{subsec:connectionstorsion}. Given $\phi \in C^\infty(M)$, we consider the generalized connection $D^\phi$ determined by the 1-form
$$
\varphi = \frac{-6}{\rk V_+ - 1} d\phi.
$$
Using the notation in Proposition \ref{prop:Ricci}, this lead us to the following formula for the generalized Ricci tensor
\begin{align*}
\operatorname{GRic}(e_2^-,e_3^+) & = \(\operatorname{Ric}^g + 2\nabla^g(d\phi) - \frac{1}{4} H \circ H - F\circ F - \frac{e^{2\phi}}{2}d^*(e^{-2\phi}H)\)(Y,Z)\\
& + i_Y c\(e^{2\phi}d_A^*(e^{-2\phi}F) + \frac{1}{2}*(F\wedge *H),t\),
\end{align*}
where $F$ denotes the curvature of the product connection given by $\nabla^T$ and $A$. The next result is a direct consecuence of this formula.

\begin{theorem}\label{th:het}
An admissible metric $V_+$ on $E$ with  free, $V_+$-compatible, generalized connection $D^\phi$ satisfies
$$
\operatorname{GRic} = 0
$$
if and only if the corresponding tuple $(g,\phi,H,A,\nabla^T)$ satisfies the field equations of heterotic supergravity \eqref{eq:dHsugra} and \eqref{eq:motionHetsugra}.
\end{theorem}

\begin{remark}
We should mention that our discussion also applies to Einstein-Yang-Mills supergravity \cite{Ch,BRWN,ChMa}, % (see also~\cite[\S 2]{BKO}),
by considering $P = P_{YM}$. The field equations of this theory are given by \eqref{eq:dHsugra} and \eqref{eq:motionHetsugra} by formally  setting $R$ equals to $0$.
\end{remark}

%The equations of motion of Heterotic supergravity originally arise by asking for two dimensional conformal invariance of a sigma model. This involves the vanishing of the so called $\beta$ functions in the physics literature \cite{CFMP}.

An interesting fact in the previous result, as well as in the generalized geometric treatment of type II supergravity \cite{CSW}, %and $11$-dimensional supergravity \cite{CSW2},
is that the three natural quantities which arise from the generalized Ricci tensor, agree exactely with the leading order term in $\alpha'$-expansion of the $\beta$ functions $\beta^G, \beta^B, \beta^A$ arising in the heterotic sigma model \cite{CFMP} (see also \cite[p.990]{QFS}) (as in \cite{CSW,CSW2}, we expect the leading term of the remaining function $\beta^\phi$ to arise from the generalized scalar curvature). The vanishing of the $\beta$ functions is the condition for conformal invariance of the associated quantum field theory, which stablishes an identification between the renormalization group flow of these physical theories with the \emph{generalized Ricci flow}. %intriguing link between generalized geometry and conformal field theory (and renormalization theory).

\end{document}